\documentclass[12pt,reqno]{article}

\usepackage[usenames]{color}
\usepackage{amssymb}
\usepackage{amsmath}
\usepackage{amsthm}
\usepackage{amsfonts}
\usepackage{amscd}
\usepackage{graphicx}

\usepackage[colorlinks=true,
linkcolor=webgreen,
filecolor=webbrown,
citecolor=webgreen]{hyperref}

\definecolor{webgreen}{rgb}{0,.5,0}
\definecolor{webbrown}{rgb}{.6,0,0}

\usepackage{color}
\usepackage{fullpage}
\usepackage{float}

\begin{document}

\theoremstyle{plain}
\newtheorem{theorem}{Theorem}
\newtheorem{corollary}[theorem]{Corollary}
\newtheorem{lemma}[theorem]{Lemma}
\newtheorem{proposition}[theorem]{Proposition}

\theoremstyle{definition}
\newtheorem{definition}[theorem]{Definition}
\newtheorem{example}[theorem]{Example}
\newtheorem{conjecture}[theorem]{Conjecture}
\newtheorem{observation}[theorem]{Observation}

\theoremstyle{remark}
\newtheorem{remark}[theorem]{Remark}

\begin{center}
\vskip 1cm{\LARGE\bf When Sets Are Not Sum-Dominant\\
\vskip 1cm}
\large
H\`ung Vi\d{\^e}t Chu\\
Department of Mathematics\\ Washington and Lee University \\
Lexington, VA 24450\\
USA \\
\href{mailto:chuh19@mail.wlu.edu}{\tt chuh19@mail.wlu.edu} \\
\end{center}
\vskip .2 in

\begin{abstract} 
Given a set $A$ of nonnegative integers, define the sum set
$$A+A = \{a_i+a_j\mid a_i,a_j\in A\}$$ and the difference set 
$$A-A = \{a_i-a_j\mid a_i,a_j\in A\}.$$ The set $A$ is said to be sum-dominant
if $|A+A|>|A-A|$. In answering a question by Nathanson, Hegarty
used a clever algorithm to find that the smallest cardinality of a
sum-dominant set is $8$. Since then, Nathanson has been asking for a
human-understandable proof of the result. We offer a computer-free proof that a set
of cardinality less than $6$ is not sum-dominant. Furthermore, we prove
that the introduction of at most two numbers into a set of numbers
in an arithmetic progression does not give a sum-dominant set. This
theorem eases several of our proofs and may shed light on future work
exploring why a set of cardinality $6$ is not sum-dominant. Finally,
we prove that if a set contains a certain number of integers from a
specific sequence, then adding a few arbitrary numbers into the set
does not give a sum-dominant set.  \end{abstract} 

\section{Introduction}

\subsection{Literature review} Given a set $A\subseteq \mathbb{N}$, define
the sum set $A+A = \{a_i+a_j \mid a_i,a_j\in A\}$ and the difference set $A-A
= \{a_i-a_j \mid a_i,a_j\in A\}$. The set $A$ is said to be
\begin{itemize}
\item {\it sum-dominant}, if $|A+A|>|A-A|$;
\item {\it balanced}, if $|A+A| = |A-A|$; and
\item {\it difference-dominant}, if $|A+A|<|A-A|$.
\end{itemize}
As addition is commutative and subtraction is not, it was
natural to conjecture that sum-dominant sets are rare. Since Nathanson's
review of the subject in 2006 \cite{Na2}, research on sum-dominant sets
has made considerable progress: see \cite{FP, Ma, Na2, Ru1, Ru2, Ru3} for
history and overview, \cite{He,MOS,MS,Na3,Zh1} for explicit constructions
, \cite{CLMS2, MO, Zh3} for positive lower bound for the percentage of
sum-dominant sets, \cite{ILMZ,MPR} for generalized sum-dominant sets,
and \cite{AMMS,CLMS1,CNMXZ,MV,Zh2} for extensions to other settings.

However, much less work has been done on how to determine whether or not a given set is sum-dominant. Only recently, Mathur and Wong \cite{MAW} gave an algorithm for checking if a set is sum-dominant. The algorithm computes and compares all pairs of possible sums and differences among numbers of the set. This paper instead focuses on certain types of not-sum-dominant sets that we can tell without the assistance of computers. 

Nathanson \cite{Na2} asked \textit{What is the smallest cardinality of a sum-dominant set?}. Hegarty \cite{He} used a clever algorithm to find $A_1 = \{0,2,3,4,7,11,12,14\}$ as the smallest sum-dominant set. Furthermore, $A_1$ is the unique sum-dominant set of cardinality $8$, up to affine transformations. However, a human-understandable proof of the result has not been produced because of the complexity lurking behind the interaction of numbers in addition and subtraction. Nathanson \cite{Na1, Na4} asked for a human-understandable proof of the smallest cardinality of a sum-dominant set. This paper proves that a set of cardinality less than $6$ is not sum-dominant without the use of computers. 
\subsection{Notation}
We introduce some notation. 
\begin{itemize}
\item Let $n,a,b\in \mathbb{N}$. Define $I_n = \{0,1,\ldots,n-1\}$ and $[a,b] = \{a,a+1,\ldots,a+\ell \mid \ell\in \mathbb{N}, \ell\le b-a\}$. So, we can write $I_n=[0,n-1]$. Define the center of $I_n$ to be $(n-1)/2$, the point that is equidistant from the two endpoints. 

\item For $(a_i)_{i=1}^\ell$ and a set $A$, we write $(a_i)_{i=1}^\ell\rightarrow A$ to mean the introduction of $\ell$ numbers $(a_i)_{i=1}^\ell$ into the set $A$ to form $A\cup\{a_i \mid  1\le i\le \ell\}$.

\item Given a set $A$ and a number $x$,
we write $x\pm A$ to mean the set $\{x\pm a_i\mid a_i\in A\}$.

\end{itemize}
\subsection{Main results}

The following is our first result. 

\begin{theorem}\label{notsum-dominant}
A set $A$ with $|A|< 6$ is not sum-dominant. 
\end{theorem}

\noindent Though our proof is concise, we are unable to prove the cases of cardinality $6$ and $7$ since the level of complexity grows quite quickly. Our main technique is to argue for a lower bound for the number of equal pairs of positive differences from $A-A$, which confines set $A$ to certain structures. 

Our next result is about the relationship between sum-dominant sets and arithmetic progressions. Since numbers from an arithmetic progression form a balanced set, it is convenient to introduce new numbers into the set (in a clever way) and produce a sum-dominant set. Indeed, the construction of sum-dominant sets using arithmetic progressions was started by Nathanson \cite{Na3} and Hegarty \cite{He}. However, little is known about the smallest number of integers needed to turn an arithmetic progression into a sum-dominant set. We prove that the introduction of at most two numbers into a set of numbers in an arithmetic progression does not give a sum-dominant set. 

\begin{theorem} \label{two}
Let $A$ be a set of numbers in an arithmetic progression and $m_1 ,m_2\in\mathbb{N}$. Then $A\cup\{m_1, m_2\}$ is not sum-dominant. 
\end{theorem}

\noindent A natural question to ask is \textit{what is the minimum number of integers to add to a set formed by numbers in an arithmetic progression and have a sum-dominant set?}. Let the number be $k$. Note that $A_1 = \{3,7,11\}\cup\{0,2,4,12,14\}$. Because $3$, $7$, $11$ is an arithmetic progression, we know that $3\le k\le 5$. Though Theorem \ref{two} is easily stated, the proof requires a clever division into cases to reduce the complexity of interactions between numbers in addition and subtraction. As a necessary condition for a set to be sum-dominant, Theorem \ref{two} provides a powerful tool to eliminate cases in arguing about the smallest cardinality of sum-dominant sets. For example, for a set of cardinality $6$, if we know that at least $4$ numbers in the set form an arithmetic progression, then the set is not sum-dominant. 

Our final result is another test for being sum-dominant and extends \cite[Theorem 1]{CNMXZ}. 
\begin{theorem}\label{three}
Let $S$ be a subset of $A$, where $A = \{a_k\}_{k=1}^\infty$ is a strictly increasing sequence of non-negative numbers. If there exists a positive integer $r$ such that
\begin{enumerate}
    \item $a_k>a_{k-1}+a_{k-r}$ for all $k\ge r+1$, and 
    \item set $A$ does not contain any sum-dominant set $S$ with $|S|\le 2r+n$ for some $n\in\mathbb{N}$, and
    \item $|S| = 2r+n+\ell$ for some $\ell\in\mathbb{N}$, 
\end{enumerate}
Then
\begin{enumerate}
    \item $S$ is not sum-dominant and $|S-S|-|S+S|\ge \ell(n+1)$,
    \item Let $m\in\mathbb{N}$ and $(b_i)_{i=1}^m$ be integers. If $ m|S| + \frac{m(m+1)}{2}\le \ell(n+1)$, then $S^* = S\cup \{b_1,\ldots,b_m\}$ is not sum-dominant. 
\end{enumerate}
\end{theorem}
\begin{example}
For the Fibonacci numbers $\{0,1, 2, 3, 5, 8, 13,\ldots\}$, we have $r = 3$. By \cite[Corollary 2]{CNMXZ}, the Fibonacci numbers have no sum-dominant subsets, so we can pick any number $n$. Let $n = 2$ and $\ell = 5$, for example. Since $m=1$ satisfies $13m+\frac{m(m+1)}{2} \le 15$, we know that a set of $13$ Fibonacci numbers and an arbitrary integer is not sum-dominant. 
\end{example}

\begin{example}
Let $\rho>\phi = \frac{1+\sqrt{5}}{2}$, the golden ratio. The geometric sequence $1,\rho,\rho ^2,\rho ^3, \ldots$ has the property that $\rho^{k}>\rho^{k-1}+\rho^{k-2}$. By \cite[Corollary 8]{CNMXZ}, this sequence has $r=2$ and has no sum-dominant subsets, so we can pick any number $n$. Let $n=2$ and $\ell=4$, for example. Since $m=1$ satisfies $10m+\frac{m(m+1)}{2}\le 12$, we know that a set of $10$ numbers from the sequence and an arbitrary integer is not sum-dominant. 
\end{example}

Section \ref{important results} proves several important results for the proof of our main theorems. Section \ref{lessthan6},  Section \ref{twointoarith}, and Section \ref{setscontainseq} prove Theorem \ref{notsum-dominant}, Theorem \ref{two}, and Theorem \ref{three}, respectively. Section \ref{final} discusses some questions for future research. 
\section{Important results} \label{important results}
We use the definition of a symmetric set in the sense of Nathanson \cite{Na3}: a set $A$ is symmetric with respect to a number $a$ if $A = a-A$. It is easy to check that any arithmetic progression is symmetric. (The sum of the two endpoints of an arithmetic progression is the point of symmetry.) The following lemma is proved by Nathanson \cite{Na3}.
\begin{lemma}
A symmetric set is balanced. \label{symnotMSTD}
\end{lemma}
\begin{proposition}\label{addnottoofar}
Let $n\in \mathbb{N}$ and $x = (n-1)+k$, where $1\le k\le n-1$. Then $x\rightarrow I_n$ gives $k$ new positive differences and $k+1$ new sums. 
\end{proposition}
\begin{proof}
We have $(I_n\cup \{x\})+(I_n\cup \{x\}) = (I_n+I_n)\cup (x+I_n)\cup\{2x\}$. Clearly, all new sums are in 
\begin{align*}
    (x+I_n)\cup \{2x\} \ =\ [(n-1)+k, (n-1)+(k+n-1)]\cup\{2x\}.
\end{align*}
Let $k = (n-1)-j$, where $j\ge 0$. Then $|((x+I_n)\cup\{2x\})\cap (I_n+I_n)| = j+1$. Hence, the number of new sums is $|(x+I_n)\cup\{2x\}|-(j+1) = n+1-(j+1) = n-j = k+1$. 

All new positive differences are in $x-I_n = \{x-(n-1),x-(n-2),\ldots,x\} = [k, k+n-1]$. Note that positive differences in $I_n-I_n$ are $[1, n-1]$. Because 
$[k, k+n-1]\backslash [1, n-1] = [n, k+n-1]$. The number of new positive differences is $|[n, k+n-1]| = k$.
\end{proof}
\begin{lemma}\label{add1toarith}
Let $x\in \mathbb{R}$. Then $\{x\}\cup I_n$ is not a sum-dominant set.
\end{lemma}
\begin{proof}
If $x\in I_n$, we are done because $I_n$ is an arithmetic progression, which is a symmetric set and thus, balanced

For $n=1$, our set has at most 2 elements, which is symmetric and thus, not sum-dominant by Lemma \ref{symnotMSTD}.

We assume $n\ge 2$. Note the number of new sums is at most $(n+1)$. We
consider the following three cases.

\bigskip

    \noindent \textbf{Case I:} $0<x<n-1$. Arrange numbers in $\{x\}\cup I_n$ in increasing order. Either to the left of $x$ or to the right of $x$, there are at least $\lceil n/2\rceil$ numbers. Without loss of generality, assume that $x$ is greater than $\lceil n/2\rceil$ numbers in $I_n$; that is, $0<1<\cdots<\lceil n/2\rceil -1 <x$. The set of new differences has $$D \ =\ \{x-(\lceil n/2\rceil -1),x-(\lceil n/2\rceil -2),\ldots,x-0\}$$ as a subset. 
    \begin{enumerate}
        \item Subcase I.1: $x-1/2\in \mathbb{Z}$. Because $0<x<n-1$, $x+x\in I_n+I_n$, implying that there are at most $n$ new sums. Since the number of new differences is at least $2|D| = 2\lceil n/2\rceil \ge n$, $I_n\cup\{x\}$ is not sum-dominant. 
        \item Subcase I.2: $x-1/2\notin \mathbb{Z}$. Then the set of new differences has $D\cup\{n-1-x\}$ as a subset. Since the number of new differences is at least $2|D\cup\{n-1-x\}| = 2(\lceil n/2\rceil +1)\ge n+2$, $I_n\cup\{x\}$ is not sum-dominant. 
    \end{enumerate}

    \noindent \textbf{Case II:} $n-1<x$. We consider two subcases. 
    \begin{enumerate}
        \item Subcase II.1: $x\notin \mathbb{N}$. The set of new differences include $x-(n-1), x -(n-2),\ldots,x-0$. Therefore, the number of new differences is at least $2n>n+1$, implying that $I\cup\{x\}$ is not sum-dominant. 
        \item Subcase II.2: $x\in \mathbb{N}$. If $x = n$, then $I_n\cup \{x\}$ have numbers from an arithmetic progression and so, is not sum-dominant. Let $x = (n-1)+k$ for some $k\in \mathbb{N}_{\ge 2}$. We have $x-I_n = \{k,k+1,\ldots, k+(n-1)\}$. Because $\max(I_n-I_n)=n-1$, if $k>n-1$, there are $n$ new positive differences and so, $I\cup \{x\}$ is not sum-dominant. If $k\le n-1$ by Proposition \ref{addnottoofar}, there are $k$ new positive differences and $k+1$ new sums, which shows that $I\cup\{x\}$ is not sum-dominant. 
    \end{enumerate}

    \noindent \textbf{Case III:} $x<0$. Due to symmetry, this case is similar to Case II. 
This completes our proof. 
\end{proof}
\begin{corollary}\label{arithnotMSTD}
A set of numbers in an arithmetic progression in union with a singleton set cannot be sum-dominant.  
\end{corollary}
\begin{proof}
Let $a\in \mathbb{N}$ and $m,d\in \mathbb{N}$. Our set is $A = \{a,a+d,\ldots, a+(n-1)d,m\}$, where $n\in\mathbb{N}$. The set $A$ is sum-dominant if and only if $\frac{1}{d}(A-a) = \{0,1,\ldots,(n-1)\}\cup\{(m-a)/d\}$ is sum-dominant. Note that $(m-a)/d$ may not be a nonnegative integer. This completes the proof. 
\end{proof}
\begin{remark}
Let $n\ge 2$, $x-1/2\notin \mathbb{Z}$ and $x\notin\{-1,n\}$. The difference set of $I_n\cup \{x\}$ has at least one more number than the sum set. This remark is evident from the proof of Lemma \ref{add1toarith}.
\end{remark}
\begin{remark}\label{not1/2onemore}
A set of numbers in an arithmetic progression is an example of a symmetric set. Though the result of Corollary \ref{arithnotMSTD} holds for any set of numbers in an arithmetic progression, it is not true for symmetric sets in general. For example, the set $A_1\backslash\{4\}$ is symmetric, and $A_1$ is sum-dominant.
\end{remark}
The following proposition offers upper bounds for the cardinality of the sum set and the difference set of a set $A$ in terms of $|A|$. The two inequalities are not hard to prove and are used by Hegarty \cite{He}.
\begin{proposition}\label{bounds}
We have the following bounds
\begin{align}
    |A+A| \ &\le \ \frac{|A|(|A|+1)}{2},\label{boundforsum}\\
    |A-A|\ &\le \ |A|(|A|-1)+1\label{boundfordiff}.
\end{align}
The equality in (\ref{boundforsum}) is achieved if the sum of any two numbers is distinct, and the equality in (\ref{boundfordiff}) is achieved if the difference between any two different numbers is distinct. 
\end{proposition}
The next observation plays a key role in reducing the complexity of our proof that a set of too small a cardinality cannot be sum-dominant. 
\begin{observation}\label{losesum}
Let $A\subseteq \mathbb{N}$ and consider $A\pm A$. Let $a_i<a_j\le a_m<a_n\in A$ such that $a_j-a_i = a_n-a_m$. If $a_j\neq a_m$, we have another pair of equal positive differences $a_m-a_i = a_n-a_j$. If $a_j = a_m$, then we do not have another pair. In both cases, we have $a_j+a_m = a_i+a_n$, a pair of equal sums. Hence, for $k$ pairs of equal positive differences, there exist at least $k/2$ pairs of equal sums. 
\end{observation}

\section{Proof of Theorem \ref{notsum-dominant}}\label{lessthan6}
For clarity, we split the proof into two parts. 
\subsection{A set $A$ with $|A|\le 4$ is not sum-dominant.}
\begin{proof}
We proceed by case analysis of $|A|$. 

If $|A| = 1$, then $|A+A| = |A-A| = 1$ and so, $A$ is not sum-dominant.

If $|A| = 2$, suppose that $A = \{a_1,a_2\}$. Then $A$ is symmetric with respect to $a_1+a_2$. By Lemma \ref{symnotMSTD}, $A$ is not sum-dominant. 

If $|A| = 3$, then $A$ is the union of an arithmetic progression with a singleton set. By Corollary \ref{arithnotMSTD}, $A$ is not sum-dominant. 

If $|A|=4$, by Proposition \ref{bounds}, $|A+A|\le 10$, while $|A-A|\le 13$. Let $k$ be the number of pairs of equal positive differences that $A-A$ has. Then, in order that $A$ is sum-dominant, 
\begin{align*}13-2k\ <\ 10-\frac{k}{2}.\end{align*}
The $k/2$ comes from Observation \ref{losesum}. We have $k\ge 3$. Therefore, $|A-A|\le 13-3\cdot 2 = 7$.
Denote $A = \{a_1,a_2,a_3,a_4\}$ with $a_1<a_2<a_3<a_4$ and $d_i = a_{i+1}-a_i$. We write out all nonnegative differences in $A-A$
\begin{align*}
    0\mbox{ }\mbox{ }\mbox{ }&d_1\mbox{ }\mbox{ }\mbox{ }d_1+d_2\mbox{ }\mbox{ }\mbox{ }d_1+d_2+d_3\\
                             &d_2\mbox{ }\mbox{ }\mbox{ }d_2+d_3\\
                             &d_3.
\end{align*}
All differences in row 1 are pairwise distinct. These nonnegative differences give 7 differences in total. Because $|A-A|\le 7$, we are not allowed to have any new differences from row 2 and row 3. Clearly, $d_2 = d_1$, which implies that $a_1,a_2,a_3$ is an arithmetic progression. By Corollary \ref{arithnotMSTD}, $A$ is not sum-dominant. 
\end{proof}

\subsection{A set $A$ with $|A|=5$ is not sum-dominant.}\label{5not}

\begin{proof}
By Proposition \ref{bounds}, we know that $|A+A|\le 15$, while $|A-A|\le 21$. Let $k$ be the number of pairs of equal positive differences. Due to Observation \ref{losesum}, we have 
$21- 2k\ < \ 15-\frac{k}{2}.$
So, $k\ge 5$ and $|A-A|\le 11$. Denote $A = \{a_1,a_2,\ldots, a_5\}$ with $a_1<a_2<\cdots<a_5$ and $d_i = a_{i+1}-a_i$. The following lists all nonnegative differences in $A-A$
\begin{align*}
    0\mbox{ }\mbox{ }\mbox{ }&d_1\mbox{ }\mbox{ }\mbox{ }d_1+d_2\mbox{ }\mbox{ }\mbox{ }d_1+d_2+d_3\mbox{ }\mbox{ }\mbox{ }d_1+d_2+d_3+d_4\\
                             &d_2\mbox{ }\mbox{ }\mbox{ }d_2+d_3\mbox{ }\mbox{ }\mbox{ }d_2+d_3+d_4\\
                             &d_3\mbox{ }\mbox{ }\mbox{ }d_3+d_4\\
                             &d_4.
\end{align*}
Differences in row 1 are pairwise distinct and account for $9$ differences in $A-A$. Because $|A-A|\le 11$, we are allowed to have at most one more difference from rows 2, 3, and 4. 

\bigskip

\noindent \textbf{Case I:} $d_2\neq d_1$. Then $$A-A\  =\ \{0,\mbox{ }d_1,\mbox{ }d_2,\mbox{ }d_1+d_2,\mbox{ }d_1+d_2+d_3,\mbox{ }d_1+d_2+d_3+d_4\}.$$
    \begin{enumerate}
        \item Subcase I.1: $d_2+d_3 = d_1$. Because $d_1>d_3$ and $d_3\in A-A$, $d_3 = d_2$. Since $d_2+d_3+d_4\in A-A$, either $d_2+d_3+d_4 = d_1+d_2$ or $d_2+d_3+d_4 = d_1+d_2+d_3$. 
        Combine the former with $d_2+d_3=d_1$ to have $d_2 = d_4$. We have $d_2=d_3=d_4$, which, by Corollary \ref{arithnotMSTD}, makes $A$ not sum-dominant. The latter gives $d_1 = d_4$, which, combined with $d_2=d_3$, makes $A$ symmetric and thus, not sum-dominant. 
        \item Subcase I.2: $d_2+d_3 = d_1+d_2$, implying $d_1=d_3$. Because $d_2+d_3+d_4\in A-A$, it must be that $d_2+d_3+d_4=d_1+d_2+d_3$ and so, $d_1=d_4$. Similarly, because $d_3+d_4\in A-A$ and $d_1 = d_3 = d_4 \neq d_2$, we must have $d_3 +d_4 = d_2$. Due to the fact that sum-dominant is preserved under affine transformations, we let $a_1 = 0$ and $d_2 = 2$ to have $A = \{0,1,3,4,5\}$. This set is not sum-dominant.
    \end{enumerate}

\noindent \textbf{Case II:} $d_2 = d_1$. If $d_2+d_3 = d_1 + d_2$, then $d_1 = d_3$. Since $d_1 = d_2 = d_3$, $A$ is not sum-dominant due to Corollary \ref{arithnotMSTD}. Therefore, $d_2+d_3\neq d_1+d_2$ or, equivalently, $d_1\neq d_3$ and 
    $$A-A \ =\  \{0,\mbox{ }d_1,\mbox{ }d_1+d_2,\mbox{ }d_2+d_3,\mbox{ }d_1+d_2+d_3,\mbox{ }d_1+d_2+d_3+d_4\}.$$
    Because $d_3\in A-A$ and $d_3 \neq d_1$, we know $d_3 = d_1+d_2$.
    Since $d_2+d_3+d_4\in A-A$ and $d_3>d_1$, it must be that $d_2+d_3+d_4 = d_1+d_2+d_3$. 
    Hence, $d_1 = d_4$ and so, $d_1 = d_2 = d_4 = d_3/2$. Due to the fact that sum-dominant is preserved under affine transformations, we let $a_1 = 0$ and $d_3 = 2$ to have $A = \{0,1,2,4,5\}$. This set is not sum-dominant.

We have shown that a set of cardinality $5$ is not sum-dominant.\end{proof}

\section{Proof of Theorem \ref{two}}\label{twointoarith}
\begin{lemma}\label{lemma1}
Let $n\in\mathbb{N}$ and two numbers $x,y\in \mathbb{R}$ such that $x\pm y$ are not integers. Then $I_n\cup\{x,y\}$ is not sum-dominant. 
\end{lemma}
\begin{proof}
If $n=1$, we know that $I_n\cup\{x,y\}$ is not sum-dominant because its cardinality is 3. So, we assume that $n\ge 2$ and $y>x$. The condition $x\pm y\notin \mathbb{Z}$ guarantees that $(y\pm I_n)\cap (x\pm I_n) = \emptyset$. For the proof, we let our original set be $I_n\cup\{x\}$ and we introduce $y$ to the set; that is, $y\rightarrow I_n\cup\{x\}$. By Lemma \ref{add1toarith}, $I_n\cup\{x\}$ is not sum-dominant. Note that the introduction of $y$ gives at most $n+2$ new sums. We consider three cases.

\bigskip

\noindent \textbf{Case I:} $0 <y< n-1$. Similar to Case I in the proof of Lemma \ref{add1toarith}, the introduction of $y$ into $I_n\cup\{x\}$ gives at least $\lceil n/2\rceil$ new positive differences. These differences result from the interaction of $y$ and $I_n$. 
\begin{enumerate} 
\item Subcase I.1: $y-1/2\notin \mathbb{Z}$. Similar to Subcase I.2 in the proof of Lemma \ref{add1toarith}, either $(n-1)-y$ or $y-0$ is another new positive difference. Hence, the number of new differences is at least $2\cdot (\lceil n/2\rceil + 1) \ge n+2$. Therefore, $I_n\cup\{x,y\}$ is not sum-dominant. 
\item Subcase I.2: $y-1/2\in \mathbb{Z}$. Due to our condition that $y\pm x\notin \mathbb{Z}$, causing $x-1/2\notin \mathbb{Z}$ and Remark \ref{not1/2onemore}, the difference set of $I\cup\{x\}$ has one more number than the sum set. Also, since $2y\in (I_n+I_n)$, the number of new sums is at most $n+1$. Because the number of new differences is at least $2\lceil n/2\rceil\ge n$, $I\cup\{x,y\}$ is not sum-dominant.  
\end{enumerate}
    
\noindent \textbf{Case II:} $n-1<y$. We have two subcases. 
    \begin{enumerate}
    \item Subcase II.1: $y\notin \mathbb{N}$. The set of new differences includes $y-(n-1),y-(n-2),\ldots,y-0$. Therefore, the number of new differences is at least $2n>n+2$, implying that $I\cup\{x,y\}$ is not sum-dominant. 
    \item Subcase II.2: $y\in \mathbb{N}$. Let $y = (n-1)+k$ for some $k\in\mathbb{N}_{\ge 2}$. The set of differences related to $y$ is $y-(I_n\cup\{x\}) \supseteq \{k,k+1,\ldots,k+(n-1)\}$. If $k>n-1$, there are $n$ new positive differences and so, $I\cup\{x,y\}$ is not sum-dominant. If $k\le n-1$, there are at least $k$ new positive differences and at most $k+2$ new sums by Proposition \ref{addnottoofar}. 
\end{enumerate}
    
\noindent \textbf{Case III:} $y<0$. Due to symmetry, this case is the same as Case II. 
We complete the proof. 
\end{proof}
\begin{lemma}\label{lemma2}
Let $n\in\mathbb{N}$ and two numbers $x,y\in \mathbb{R}$ such that $x \pm y$ are integers. Then $I_n\cup\{x,y\}$ is not sum-dominant. 
\end{lemma}
\begin{proof}
The proof is divided into four parts, each of which deals with a specific position of $x$ and $y$ when being introduced to $I_n$.

\begin{center}\textbf{Part I. $n-1<x<y$}\end{center} We know that $\{x,y\}\rightarrow I_n$ gives at most $2n+3$ new sums. 

    \noindent \textbf{Case I:} $x,y\notin \mathbb{N}$. Then $x,y\rightarrow I_n$ gives the following $n+1$ new differences $D = \{x-(n-1),x-(n-2),\ldots,x-0\}\cup\{y\}$. Therefore, in order that $I_n\cup\{x,y\}$ is sum-dominant, $x,y\rightarrow I_n$ must give exactly $2n+3$ new sums. 

    \begin{enumerate}
    \item Subcase I.1: $y-x\in\mathbb{N}$. To have $2n+3$ new sums, $[y+0,y+(n-1)]$ and $[x+0,x+(n-1)]$ must be disjoint, implying that $y-(n-1)>x$. So, $y-(n-1)$ is a new difference that is not in $D$. Hence, we have $n+2$ new positive differences and so, $I\cup\{x,y\}$ is not sum-dominant.
    \item Subcase I.2: $y+x\in\mathbb{N}$. 
    \begin{itemize}
        \item If $x-1/2\notin \mathbb{Z}$, then $y-1$ is a new positive difference not in $D$ and we have $n+2$ new positive differences. 
        \item If $x-1/2\in \mathbb{Z}$, then note that if $y>x+1$, we have $y-1$ is a new difference not in $D$. So, $y\le x+1$. Because $x-1/2\in\mathbb{Z}$ and $x+y\in\mathbb{N}$, $y = x+1$. Hence, we have a pair of equal sums and so, the number of new sums is at most $2n+2$. Therefore, $I\cup\{x,y\}$ is not sum-dominant.  
    \end{itemize}
    \end{enumerate}
    
    \noindent \textbf{Case II:} $x,y\in \mathbb{N}$. Write $x = (n-1)+k$ for some $k\ge 2$ and $y = x+j$ for some $j\ge 1$. Note $x-I_n = \{k,k+1,\ldots,k+(n-1)\}$. If $k>n-1$, then all numbers in $x-I_n$ are new differences. Because $y>x = k+(n-1)$, $y$ is another new difference. So, the number of new differences is at least $2(n+1) = 2n+2$. Therefore, in order that $I_n\cup\{x,y\}$ is sum-dominant, $x,y\rightarrow I_n$ must give exactly $2n+3$ new sums. Similar to Subcase I.1 above, we can show that $I\cup\{x,y\}$ is not sum-dominant. Hence, $k\le n-1$.
    
    If $j>n-1$, then $y-(I_n\cup\{x\})$ is the set of $n+1$ new positive differences. In order that $I\cup\{x,y\}$ is sum-dominant, $x,y\rightarrow I_n$ must gives exactly $2n+3$ new sums. This is a contradiction because $k\le n-1$, causing $x+0\in [0,2n-2]$. Hence, $j\le n-1$. 
    
    We have shown that $j,k\le n-1$. Suppose that $k = j = n-1$. It can be easily checked that $I_n\cup\{2n-2,3n-3\}$ is not sum-dominant. So, either $k\le n-2$ or $j\le n-2$. 
    \begin{itemize}
    \item If $k\le n-2$, we know that $x+0 = (n-1)+k \le 2n-3$ and $x+1 = (n-1)+k+1\le 2n-2$. Also, because $j\le n-1$, $x+0\le y+0\le x+(n-1)$. So, the number of new sums is at most $2n$. 
    
    \item If $j\le n-2$, then $x\le y+0\le  y+1\le (x+j)+1\le x+(n-1)$. Also, because $k\le n-1$, $x+0\le (n-1)+k\le 2n-2$. The number of new sums is also at most $2n$. 
    \end{itemize}
    We have shown that the number of new sums is at most $2n$. Suppose that $k+j>n-1$. Because $y = (n-1)+(k+j)$, $y\rightarrow I_n$ gives $n$ new positive differences. This implies that $I_n\cup\{x,y\}$ is not sum-dominant. So, $k+j\le n-1$. We consider the sum set of $I_n\cup\{x,y\}$. Note that $\min (x+I_n)  < \min (y+I_n) = (n-1)+(k+j)\le 2n-2$, so the set of new sums is a subset of $((y+I_n)\backslash (I_n+I_n))\cup\{x+y,2x,2y\}$, which is at most $k+j+3$ due to Proposition \ref{addnottoofar}. Indeed, Proposition \ref{addnottoofar} states that since $y = (n-1) + k+ j$ and $k+j\le n-1$, $y\rightarrow I_n$ gives $k+1$ new sums (including $2y$). Also by Proposition \ref{addnottoofar}, $y\rightarrow I_n$ gives $k+j$ new positive differences. Since $2(k+j)\ge k+j+3$, $I_n\cup\{x,y\}$ is not sum-dominant. 

\begin{center}\textbf{Part II.} $0<x<n-1<y$\end{center}
    
   \noindent \textbf{Case I:} If $x-1/2\notin \mathbb{Z}$, there are $n$ new positive differences from the interaction between $x$ and $I_n$. Since $y-0>n-1$, $y$ is another new positive difference. If $x+y\in\mathbb{N}$, $y-1$ is another positive difference. Hence, we do not have a sum-dominant set because we have at most $2n+3$ new sums. Therefore, in order that $I_n\cup\{x,y\}$ is sum-dominant, $y-x\in\mathbb{N}$ and $x,y\rightarrow I_n$ must give exactly $2n+3$ new sums, which implies that $(x+[0,n-1])\cap (y+[0,n-1]) = \emptyset$. So, $y>(n-1)+x$. This inequality leads to the following two cases.  
   \begin{itemize}
       \item If $x>1$, then $y>n$ and there is another new difference $y-1>n-1$.
       \item If $0<x<1$, then $n-1<y<n$. Hence, $n-1<x+y<n+1$ and $n-2<y-x<n$. Because $y-x\in\mathbb{N}$, $y-x = n-1$ or $y+0 = x+(n-1)$. Since we have a pair of equal sums, the number of new sums is less than $2n+3$. 
   \end{itemize}

    \noindent \textbf{Case II:} If $x-1/2\in\mathbb{Z}$, the following are $n$ new positive differences $y-(n-1),y-(n-2),\ldots,y-0$ from $x,y\rightarrow I_n$. Hence, the number of new differences is at least $2n$. Because $2x\in\mathbb{N}$ and $0<2x<2n-2$, the number of new sums is at most $2n+2$. In order that $I\cup\{x,y\}$ is sum-dominant, there must not be a new positive difference other than the $n$ differences above.
    So, the positive difference $(n-1)-x$ must be equal to $y-(n-i)$ for some $1\le i\le n$ and so, $x+y = 2n-1-i$.
    Similarly, the positive difference $x-0$ must be equal to $y-(n-j)$ for some $1\le j\le n$ and so, $x = y-(n-j)$ or equivalently, $x+(n-j) = y+0$. Hence, we have at least two pairs of equal sums. Therefore, the number of new sums is at most $2n$. This shows that $I\cup\{x,y\}$ is not sum-dominant. 
    
\begin{center}\textbf{Part III.} $0<x<y<n-1$\end{center}
Let $k$ be the number such that $k<x<k+1$.

\noindent \textbf{Case I:} $x+\frac{1}{2} \notin \mathbb{Z}$.
\begin{enumerate}
    \item Subcase I.1: $y-x\in\mathbb{N}$. New positive differences include
    \begin{align*}
        x-k &\ <\ x-(k-1) \ <\ \cdots \ <\ x-0 \ <\ y-0\\
        (k+1)-x&\ <\ (k+2)-x \ <\ \cdots \ < \ (n-1)-x.
    \end{align*}
    Because $x+\frac{1}{2}\notin\mathbb{Z}$, numbers in the two rows are pairwise distinct. We have these $n+1$ new positive differences. Since $0 < y-x\in\mathbb{N}$, there exists $0< i < n-1$ such that $y-x = i$. So, $y+0 = x+i$, implying that we have at most $2n+2$ new sums. Therefore, $I\cup\{x,y\}$ is not sum-dominant. 
    \item Subcase I.2: $y+x\in\mathbb{N}$. Because $0<x+y<2n-2$, we have at most $2n+2$ new sums. 
    \begin{itemize}
    \item If $x,y$ lie on the same side of the center of $I_n$, we assume that $0<(n-1)/2<x<y<n-1$. The following are new differences
    \begin{align}  
    x-k&\ <\ \cdots \ < \ x-1\ <\ x-0 \label{newdiff1}\\
    (k+1)-x &\ <\ \cdots\ <\ (n-2)-x\ <\ (n-1)-x\ <\ y-0.\label{newdiff2}\end{align}
    Because $x+\frac{1}{2}\notin \mathbb{Z}$, the numbers in these two rows are pairwise distinct. These are $n+1$ new positive differences. Hence, $I\cup\{x,y\}$ is not sum-dominant. 
    \item If $x,y$ lies on two sides of the center of $I_n$, we assume that $x$ is closer to the center than $y$. (If they are equidistant from the center, we have a balanced set with the same center of $I_n$). The numbers in Row \ref{newdiff1} and Row \ref{newdiff2} are still new positive differences. Hence, $I_n\cup\{x,y\}$ is not sum-dominant. 
    \end{itemize}
\end{enumerate}

\noindent \textbf{Case II:} $x+\frac{1}{2}\in \mathbb{Z}$. We write $y = x+ j$ for some $1\le j<n-1-x$.
\begin{enumerate}
    \item Subcase II.1: $x,y$ lie on one side of the center of $I_n$. Assume that $0<\frac{n-1}{2}\le x<x+j<n-1$. We have
    \begin{align*}
        x+[0,n-1] &\ =\ [x,x+n-1],\\
        y+[0,n-1]&\ =\ [y,y+(n-1)] \ =\ [x+j, x+j+(n-1)].
    \end{align*}
    Because $0<2x,2y, x+y<2n-2$, they are not new sums and so, all new sums are $[x,x+n-1]\cup[x+j,x+j+(n-1)]$, which contains $n+j$ numbers. The following are new positive differences
    \begin{align*}
        x-k \ <\ x-(k-1)\ <\ \cdots\ <\ x-0\ <\ (x+j)-(j-1) \ <\ \cdots \ <\ (x+j)-0.
    \end{align*}
    So, there are at least $2(j+k+1)$ new differences. Note that $$2(j+k+1) \ =\ 2j+2(k+1)\ >\ 2j+2x\ \ge \ 2j+2\cdot \frac{n-1}{2} \ =\  2j+n-1 \ \ge\ n+j.$$
    Hence, $I_n\cup\{x,y\}$ is not sum-dominant. 
    \item Subcase II.2: $x,y$ lie on two side of the center of $I_n$. As above, the number of new sums is still $n+j$. Without loss of generality, assume that $(x+j)-(n-1)/2<(n-1)/2-x$; that is, $x+j$ is closer to the center. We have
    \begin{align*}
        (x+j) - [0,k+j] &\ =\ \{x-k, x-k+1,\ldots, x+j\},\\
        [x+(x+j)+1,n-1] - x &\ =\{x+j+1,\ldots,n-1-x\}.
    \end{align*}
    Because 
    $$x-k\ <\ x-k+1\ <\ \cdots\ <\ x+k<x+j+1\ <\ \cdots\ <\ n-1-x,$$
    all these $n-2x+k$ numbers are distinct positive differences and we have $2(n-2x+k) = 2(n-k-1)$ new differences. Since $2(n-k-1)>n+j$ ($x+j$ is closer to the center), $I\cup\{x,y\}$ is not sum-dominant.  
\end{enumerate}
    
\begin{center}\textbf{Part IV.} $x<0<n-1<y$\end{center}
If $y-(n-1)=0-x$, we have a symmetric set, which is not sum-dominant. Without loss of generality, assume that $y-(n-1)>0-x$. Note $\{x,y\}\rightarrow I_n$ gives at most $2n+3$ new sums. If we can show that the number of new differences is at least $n+2$, then the new set is not sum-dominant. 

\noindent \textbf{Case I:} $x,y\notin \mathbb{N}$.  New positive differences include 
    $$0-x\ <\ y-(n-1)\ <\ y-(n-2)\ <\ \cdots\ <\ y-0\ <\ y-x.$$

\noindent \textbf{Case II:} $x,y\in\mathbb{N}$. Let $x = -j$ and $y=(n-1)+k$. Because $y-(n-1)>0-x$, $k>j$.
\begin{enumerate}
    \item Subcase II.1: $k>n-1$ and $j>n-1$. We have $y - [0,n-1] = [k,k+(n-1)]$ and $y-x = (n-1)+k+j$. Because $$(n-1)<j<k<k+1<\cdots<k+(n-1)<(n-1)+k+j,$$ there are at least $n+2$ new positive differences $\{j,(n-1)+k+j\}\cup [k,k+(n-1)]$. Hence, $I\cup\{x,y\}$ is not sum-dominant. 
    \item Subcase II.2: $k>n-1$ and $j\le n-1$. Then $0\le -j+(n-1) = x+(n-1)\le (n-1)$. So, $x+(n-1)$ is not a new sum and so, we have at most $2n+2$ new sums. As above, there are at least $n+1$ new positive differences $\{(n-1)+k+j\}\cup [k,k+(n-1)]$. Hence, $I\cup\{x,y\}$ is not sum-dominant. 
    \item Subcase II.3: $j<k\le n-1$. We will find all new sums from $x+[0,n-1]$ and $y+[0,n-1]$. We have
    \begin{align*}
        x + [0,n-1] &\ =\ [-j, n-1-j] \ =\ [-j,-1]\cup [0,n-1-j],\\
        y + [0,n-1]&\ =\ [n-1+k,2n-2+k]\ =\ [n-1+k,2n-2]\cup [2n-1,2n-2+k]. 
    \end{align*}
    New sums are from $[-j,-1]\cup [2n-1,2n-2+k]\cup\{2x,2y\}$. (We do not count $x+y$ because $0\le x+y = (n-1)+k-j\le 2n-2$.) So, there are at most $j+k+2$ new sums. We find a lower bound for the number of new differences by looking at 
    \begin{align*}
        y-[0,n-1]\ =\ [y-(n-1),y] \ =\ [k,(n-1)+k] \ =\ [k,(n-1)]\cup [n,n+(k-1)]. 
    \end{align*}
    The $k$ numbers in $[n,n+(k-1)]$ are new positive differences. Another new difference is $y-x$. So, the number of new differences is bounded below by $2(k+1)=2k+2$. This is bigger than $j+k+2$. Hence, $I\cup\{x,y\}$ is not sum-dominant.  
\end{enumerate}
As these four parts consider all possible cases, we have completed the proof. \end{proof}
Due to the fact that being sum-dominant is preserved under affine transformation, Theorem \ref{two} follows immediately from Lemma \ref{lemma1} and Lemma \ref{lemma2}. 

\begin{remark}
We can use Theorem \ref{two} to simplify the proof of Theorem \ref{notsum-dominant}. 
\begin{itemize}
    \item For $|A| = 4$, we write $A=\{a_1,a_2,a_3,a_4\}$. Because $a_1,a_2$ is an arithmetic progression, adding $a_3,a_4$ to the set of $\{a_1,a_2\}$ does not give a sum-dominant set.
    \item For $|A|=5$, if $A$ contains an arithmetic progression of length $3$, then we are done. We look at Section \ref{5not} and easily eliminate Subcase I.1 and Case II, thus shortening the proof. 
\end{itemize}
The trade off is that the proof of Theorem \ref{two} is much more computationally involved. 
\end{remark}
\section{Proof of Theorem \ref{three}}\label{setscontainseq}
\begin{proof}
We first prove item 1. Let $S' =\{s_1,s_2,\ldots,s_{k-1}\}$, where $s_1<s_2<\cdots<s_{k-1}$ be a finite subset of $A$. We show that $S = \{s_1,s_2,\ldots,s_{k-1},s_k\}$ with $s_k>s_{k-1}$ is not a sum-dominant set by induction. If $k\le 2r+n$, then $S$ is not sum-dominant by the second assumption of the theorem. If $k\ge 2r+n+1$, consider the number of new sums and differences obtained by adding $s_k$. The number of new sums is at most $k$. We prove that the number of new differences is at least $k+n+1$.

Since $k\ge 2r+n+1$, we have $k-\lfloor \frac{k+n+2}{2}\rfloor\ge r$. Let $t = \lfloor \frac{k+n+2}{2}\rfloor$. Then, $t\le k-r$, which implies $s_t\le s_{k-r}$. The largest difference between elements in $S'$ is $s_{k-1}-s_1$; we now show that we have added at least $t$ distinct differences greater than $s_{k-1}-s_1$. Denote $s_i = a_{g(i)}$, where $g: \mathbb{Z}^+\rightarrow \mathbb{Z}^+$ is an increasing function. We have
\begin{align*}
    s_k-s_t&\ \ge \ s_k-s_{k-r} \ =\ a_{g(k)} - a_{g(k-r)}\\
    &\ \ge \ a_{g(k)} - a_{g(k)-r}\\
    &\ >\ a_{g(k)-1} - a_1\\
    &\ \ge \ s_{k-1}-a_1 \ \ge\ s_{k-1}-s_1. 
\end{align*}
The third inequality is due to the first assumption on $\{a_n\}$. Since $s_k-s_t > s_{k-1}-s_1$, we know that 
$$s_{k}-s_t, \ldots, s_{k}-s_2, s_{k} - s_1$$
are $t$ differences greater than the greatest difference in $S'$. As we could subtract in the opposite order, the number of new differences obtained from adding $s_k$ to $S'$ is at least
$$2t \ =\ 2\bigg\lfloor \frac{k+n+2}{2}\bigg\rfloor\ \ge\ k+n+1.$$
We have seen that by adding $s_k$ to the set $S'$, the number of differences goes by at least $n+1$ more than the number of sums. Note that $S'$ must have at least $2r+n$ elements. 

Next, we prove item 2. To form $S^*$, we first add $b_1$ to the set, which gives at most $|S|+1$ new sums. When we add $b_2$ to $S^*\cup\{b_1\}$, we have at most $|S|+2$ new sums. Continue the process until we have added all $m$ numbers $b_1,b_2,\ldots,b_m$ to the set, the number of new sums is at most $\sum_{i=1}^m(|S|+i) = m|S|+\frac{m(m+1)}{2}$. Our original set $S$ has $|S-S|-|S+S| = \ell(n+1)$; therefore, if $m|S|+\frac{m(m+1)}{2}< \ell(n+1)$, then $|S^*-S^*|-|S^*+S^*| \ge |S-S|-(|S+S|+ m|S|+\frac{m(m+1)}{2})\ge 0$. Therefore, $S^*$ is not sum-dominant. 
\end{proof}

\section{Future work}\label{final}
We list some natural topics for future research. 

\begin{itemize}
    \item Is there a human-understandable proof that $6$ is not the cardinality of a sum-dominant set? 
    \item What is the smallest number of integers added to a set of numbers in an arithmetic progression to have a sum-dominant set?  
    \item Is the following conjecture true? 
    \begin{conjecture}
    The union of two arithmetic progressions is not a sum-dominant set. 
    \end{conjecture}
    The case analysis in this paper is unable to solve this general case because the complexity grows quickly when more numbers are added to an arithmetic progression. The case that the minimum of one arithmetic progression is bigger than the maximum of the other is easy. If this conjecture is true, our Theorem \ref{two} follows immediately because the set of two numbers forms an arithmetic progression. 
\end{itemize}
\section{Acknowledgments}
First, I would like to thank professor Steven Miller at Williams College for introducing me to this interesting topic. Next, many thanks to the anonymous referee for valuable comments which helped to improve the article. Thanks to professor Kevin Beanland at Washington and Lee University for proofreading this paper. Finally, I would like to dedicate this paper to my parents and my brother for their support throughout my undergraduate years at Washington and Lee University.

\bigskip
\hrule
\bigskip

\noindent 2010 {\it Mathematics Subject Classification}:
Primary 11P99; Secondary 11K99.

\noindent \emph{Keywords: } sum-dominant set, MSTD set.

\bigskip
\hrule
\bigskip

\vspace*{+.1in}
\noindent
Received March 6 2019;
revised versions received May 17 2019; May 20 2019.
Published in {\it Journal of Integer Sequences}, May 21 2019.

\bigskip
\hrule
\bigskip

\end{document}